\documentclass{amsart}

\usepackage{amsmath, amsthm, amssymb, amstext, amsfonts}
\usepackage{enumerate}
\usepackage{graphicx}
\usepackage[applemac]{inputenc}

\theoremstyle{plain}

\newtheorem{thm}{Theorem}[section]
\newtheorem{lem}[thm]{Lemma}

\newtheorem{prop}[thm]{Proposition}

\newtheorem*{prob*}{Problem}

\theoremstyle{definition}

\newtheorem{rem}[thm]{Remark}

\newcommand{\cV}{\ensuremath{\mathcal{V}}}

\newcommand{\sm}{\ensuremath{\smallsetminus}}

\newcommand{\isom}{\ensuremath{\cong}}
\newcommand{\inv}{\ensuremath{^{-1}}}

\newcommand{\Aut}{\textnormal{Aut}}

\newcommand{\es}{\ensuremath{\emptyset}}

\newcommand{\sub}{\subseteq}

\def\td{tree-decom\-po\-si\-tion}

\newcommand{\nat}{{\mathbb N}}
\newcommand{\real}{{\mathbb R}}

\newcommand{\ganz}{{\mathbb Z}}

\newcommand{\AF}{\ensuremath{\mathcal A}}

\newcommand{\CF}{\ensuremath{\mathcal C}}
\newcommand{\DF}{\ensuremath{\mathcal D}}
\newcommand{\EF}{\ensuremath{\mathcal E}}
\newcommand{\FF}{\ensuremath{\mathcal F}}

\newcommand{\HF}{\ensuremath{\mathcal H}}

\newcommand{\VF}{\ensuremath{\mathcal V}}
\newcommand{\WF}{\ensuremath{\mathcal W}}

\newenvironment{txteq*}
  {
    \begin{equation*}
    \begin{minipage}[c]{0.85\textwidth} 
    \em                                
  }
  {\end{minipage}\end{equation*}\ignorespacesafterend}

\begin{document}

\title{Planar transitive graphs}
\author{Matthias Hamann}
\address{Matthias Hamann, Department of Mathematics, University of Hamburg, Bundes\-stra\ss e~55, 20146 Hamburg, Germany}
\date{}
\maketitle

\begin{abstract}
We prove that the first homology group of every planar locally transitive finite graph~$G$ is a finitely generated $\Aut(G)$-module and we prove a similar result for the fundamental group of locally finite planar Cayley graphs.
Corollaries of these results include Droms's theorem that planar groups are finitely presented and Dunwoody's theorem that planar locally finite transitive graphs are accessible.
\end{abstract}

\section{Introduction}\label{sec_Intro}

A finitely generated group is \emph{planar} if it has some locally finite planar Cayley graph.
Droms~\cite{D-InfiniteEndedGroups} proved that finitely generated planar groups are finitely presented.
In this paper, we shall present an alternative proof of his result.
Whereas Droms's proof uses an accessibility result of Maskit~\cite{Maskit65} for planar groups, our self-contained proof does not.
Instead, our proof will be based on a general result about planar graphs~-- a variant of Theorem~\ref{thm_canGen} below.
But before we state that theorem, we have to make some definitions first.

We call a graph \emph{finitely separable} if no two distinct vertices are joined by infinitely many edge disjoint paths, or equivalently, any two vertices are separable by finitely many edges.

Let $G$ be a planar graph with planar embedding $\varphi\colon G\to\real^2$.
Two cycles $C_1,C_2$ in~$G$ are \emph{nested} if no $C_i$ has vertices or edges in distinct faces of~$\varphi(C_{3-i})$.
A~set of cycles is \emph{nested} if every two of its elements are nested.
We shall prove the following two theorems.

\begin{thm}\label{thm_canGen}
Every $3$-connected finitely separable planar graph has a canonical nested set of cycles generating the first homology group.
\end{thm}

Here, \emph{canonical} means mostly that the set of cycles is invariant under the automorphisms of the graph.
But in addition, our proof is constructive and this construction commutes with graph isomorphisms, i.\,e.\ whenever we run this construction for two isomorphic graphs $G$ and~$H$, then this isomorphism maps the set of cycles in~$G$ we obtain to that of~$H$.

\begin{thm}\label{thm_finGenIntro}
Every planar locally finite transitive graph $G$ has a set of cycles that generates the first homology group and consists of finitely many $\Aut(G)$-orbits.
\end{thm}

Note that Theorems~\ref{thm_canGen} and~\ref{thm_finGenIntro} are easy to prove if the graph has no accumulation points in the plane, i.\,e.\ if it is \emph{VAP-free}, as you may then take the finite face boundaries as generating set, see e.\,g.\ \cite[Lemma 3.2]{G-PlanarCayleyGraphsAndComplexes}.

In addition to Theorems~\ref{thm_canGen} and~\ref{thm_finGenIntro}, we prove their analogues for the set of all closed walks instead of the first homology group.
From this variant of Theorem~\ref{thm_finGenIntro} for closed walks, we shall deduce the following theorem about fundamental groups of planar Cayley graphs, in which, for a set~$S$, we denote by $\FF_S$ the free group with free generating set~$S$.

\begin{thm}\label{thm_main}
Let $G$ be a locally finite planar Cayley graph of a finitely generated group~$\Gamma=\langle S\mid R\rangle$.
Then the fundamental group of~$G$ is a finitely generated $\FF_S$-module.
\end{thm}

Note that Droms's theorem about the finite presentability of planar groups follows directly from Theorem~\ref{thm_main}.

\medskip

Theorem~\ref{thm_canGen} has various analogues in the literature:
in~\cite{H-GeneratingCycleSpace} the author proved the corresponding result for the cycle space\footnote{The \emph{cycle space} of a graph is the set of finite sums of edge sets of cycles over~$\mathbb{F}_2$.} of $3$-connected finitely separable planar graphs and, previously, Dicks and Dunwoody~\cite{DiDu-GroupsGraphs} proved the analogous result for the cut space\footnote{The \emph{cut space} of a graph is the set of finite sums  over $\mathbb F_2$ of minimal separating edge sets.} of arbitrary graphs.

The mentioned theorem of Dicks and Dunwoody is one of the central theorems for the investigation of transitive graphs with more than one end and hence of accessible graphs and of accessible groups.
(We refer to Section~\ref{sec_Access} for definitions.)
Even though, accessibility has a priori more in common with the cut space than with the cycle space or the first homology group, the main result of~\cite{H-Accessibility} exhibited a connection between accessibility and the cycle space:

\begin{thm}\label{thm_accessMain}\cite{H-Accessibility}
Every transitive graph $G$ whose cycle space is a finitely generated $\Aut(G)$-module is accessible.
\end{thm}

As an application of our results and Theorem~\ref{thm_accessMain} we shall obtain Dunwoody's~\cite{D-PlanarGraphsAndCovers} theorem that locally finite transitive planar graphs are accessible.

The proofs for Theorems~\ref{thm_canGen} and~\ref{thm_finGenIntro} and their variants for closed walks are very similar.
That is, why we present only the proof for the more involved case of closed walks and then discuss in Section~\ref{sec_hom} the situation for the first homology group.
Readers that are familiar with the first homology group can already verify during the reading of the first sections that the results there stays valid for the first homology group.

\section{Indecomposable closed walks}

The \emph{sum} of two walks $W_1,W_2$ where $W_1$ ends at the starting vertex of~$W_2$ is their concatenation.
Let $W=x_1x_2\ldots x_n$ be a walk.
By $W\inv$ we denote its \emph{inverse} $x_n\ldots x_1$.
For $i<j$, we denote by $x_iWx_j$ the subwalk $x_i\ldots x_j$.
If $x_{i-1}=x_{i+1}$ for some~$i$, we call the walk $W':=x_1\ldots x_{i-1}x_{i+2}\ldots x_n$ a \emph{reduction} of~$W$.
Conversely, we \emph{add} the \emph{spike} $x_{i-1}x_ix_{i+1}$ to~$W'$ to obtain~$W$.
If $W$ is a closed walk, we call $x_i\ldots x_nx_1\ldots x_{i-1}$ a \emph{rotation} of~$W$.
By $\WF(G)$ we denote the set of all closed walks.

Let $\VF$ be a set of closed walks.
The smallest set $\VF'\supseteq \VF$ of closed walks that is invariant under taking sums, reductions and rotations and under adding spikes is the set of closed walks \emph{generated by $\VF$}.
We also say that any $V\in\VF'$ is \emph{generated by~$\VF$}.
A closed walk is \emph{indecomposable} if it is not generated by closed walks of strictly smaller length.
Note that no indecomposable closed walk $W$ has a \emph{shortcut}, i.\,e.\ a (possibly trivial) path between any two of its vertices that has smaller length than any subwalk of any rotation of~$W$ between them.
Indeed, let~$P$ be a shortest shortcut of~$W$ and $Q_1,Q_2$ be two subwalks of~$W$ whose end vertices are those of~$W$ and whose concatenation is~$W$.
Then $Q_1 P$ and $P\inv Q_2$ sum to a closed walk that has $W$ as a reduction.
As shortcuts may be trivial, we immediately obtain the following.

\begin{rem}\label{rem_indecCW=C}
Every indecomposable closed walk is a cycle.
\end{rem}

Let $G$ be a planar graph.
The {\em spin} of a vertex $x\in V(G)$ is the cyclic order of the set of edges incident with~$x$ in clockwise order.
Let $R=x_0\ldots x_\ell$ and $W=y_1\ldots y_\ell$ be two walks in a planar graph~$G$ such that $x_i=y_i$ for all $1\leq i\leq \ell-1$.
We call $R$ a \emph{crossing} of~$W$ if one of the following holds:
\begin{enumerate}[(i)]
\item the edges $x_0x_1,x_1x_2,y_0x_1$ are contained in this order in the spin of~$x_1$ and $x_{\ell-2}x_{\ell-1},x_{\ell-1}y_\ell,x_{\ell-1}x_\ell$ are contained in this order in the spin of~$x_{\ell-1}$;
\item the edges $y_0x_1,x_1x_2,x_0x_1$ are contained in this order in the spin of~$x_1$ and $x_{\ell-2}x_{\ell-1},x_{\ell-1}x_\ell,x_{\ell-1}y_\ell$ are contained in this order in the spin of~$x_{\ell-1}$.
\end{enumerate}
These crossing are shown in Figure~\ref{fig_cross}.
Note that this definition is symmetric in~$R$ and~$W$.
So $R$ is a crossing of~$W$ if and only if $W$ is a crossing of~$R$.

\begin{figure}[h]
\begin{center}
\mbox{}\hfill
\includegraphics[width=.35\textwidth]{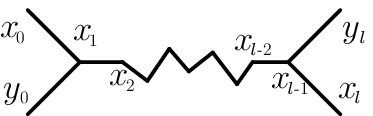}\hfill
\includegraphics[width=.35\textwidth]{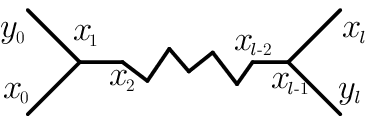}
\hfill
\mbox{}\caption[Figure 1]{The two possible crossings}\label{fig_cross}
\end{center}\end{figure}

For a closed walk $W$ and $n\in\nat$ let $W^n$ be the $n$-times concatenation of~$W$.
Two closed walks $R$ and $W$ \emph{cross} if there are $i,j\in\nat$ such that $R^i$ contains a crossing of a subwalk of~$W^j$.
They are \emph{nested} if they do not cross.

\begin{lem}\label{lem_indecomp}
Let $G$ be a planar graph and let $W_1,W_2\sub G$ be two indecomposable closed walks of lengths $n_1,n_2$, respectively.
Let $P_1\sub W_1$ be a non-trivial subwalk of shortest length that meets $W_2$ in precisely its end vertices.
Let $P_2\sub W_2$ be a shortest walk with the same end vertices as~$P_1$.
Then one of the following is true.
\begin{enumerate}[\rm (i)]
\item\label{itm_indecomp1} $|P_1|=|P_2|$ and $P_2$ meets $W_1$ only in its end vertices;
\item\label{itm_indecomp2} $|P_1|\geq|P_2|$ and $P_1 P_2\inv$ is a rotation of~$W_1$;
\item\label{itm_indecomp3} $|P_1|\geq|P_2|$ and $(W_1-P_1)P_2$ is a rotation of~$W_2$ or $W_2\inv$.
\end{enumerate}
\end{lem}

\begin{proof}
Let $v,w$ be the end vertices of~$P_1$ and recall from Remark~\ref{rem_indecCW=C} that $W_1$ and $W_2$ are cycles.
Let $Q_1$ and~$Q_2$ are the two subpaths of~$W_2$ with end vertices $v$ and~$w$.\footnote{Strictly speaking, one is just a subwalk of a rotation of the reflection of~$C_2$.}
Then $W_2$ is a reduction of a rotation of the sum of $vP_1wQ_1\inv v$ and $vQ_2wP_1\inv v$.
First, assume $|P_1|<|P_2|$.
By the choice of~$P_2$, we have $|P_2|\leq|Q_1|$ and $|P_2|\leq|Q_2|$, so $P_1$ is a shortcut of~$W_2$, which is impossible.
Hence, we have $|P_1|\geq|P_2|$.

If $P_2$ is a subwalk of~$W_1$, then we directly have that $P_1P_2$ is a rotation of~$W_1$ and~(\ref{itm_indecomp2}) holds.
So we may assume that $P_2$ contains an edge outside of~$W_1$.

Let us suppose that $P_2$ has an inner vertex on~$W_1$.
So any subwalk $xP_2y$ that intersects $W_1$ in precisely its end vertices has shorter length than~$P_2$ and hence has shorter length than~$P_1$.
Note that such a subpath exists as~$P_2$ has an edge outside~$W_1$.
But $xP_2y$ cannot be a shortcut of~$W_1$.
So the distance between $x$ and~$y$ on~$W_1$ is at most $|xP_2y|$.
The subpath $Q$ of~$W_1$ realising the distance of~$x$ and~$y$ on~$W_1$ together with $xP_2y$ does not contain~$v$ and~$w$.
So it cannot be~$W_2$.
As $W_2$ is a cycle, some edge of~$Q$ does not lie on~$W_2$ and hence~$Q$ contains some subwalk that contradicts the choice of~$P_1$.

So $P_2$ meets $W_1$ only in its end vertices.
Then $W_1$ is a reduction of the sum of $(W_1-P_1)P_2$ and $P_2\inv P_1$.
As $W_1$ is indecomposable, $P_2$ is not a shortcut of~$W_1$ and thus we have either $|P_2|=|P_1|$ or $|P_2|=|W_1-P_1|$.
The first case implies (\ref{itm_indecomp1}) while, if the first case does not hold, we have $|P_1|>|P_2|=|W_1-P_1|$.
Thus, the minimality of~$|P_1|$ implies that $W_1-P_1$ lies on~$W_2$.
So we have that $(W_1-P_1) P_2$ is a rotation of~$W_2$  or $W_2\inv$ as~$P_2$ meets $W_1$ only in its end vertices.
This shows (\ref{itm_indecomp3}) in this situation.
\end{proof}

If $C$ is a cycle in a planar graph~$G$, we denote by $f_C^0$ the bounded face of~$C$ and by $f_C^1$ the unbounded face.

For two closed walks $C,D$ of~$G$, we call a non-trivial maximal subwalk $P$ of~$C$ that has precisely its end vertices in~$D$ a \emph{$D$-path in~$C$}.
By $n(C,D)$ we denote the number of $C$-paths in~$D$.

\begin{lem}\label{lem_nestedCorners}
Let $G$ be a planar graph and let $C,D\sub G$ be two indecomposable closed walks.
Then there are nested indecomposable closed walks $\widetilde{C}$ and $\widetilde{D}$ with $|C|=|\widetilde{C}|$ and $|D|=|\widetilde{D}|$ that are either the boundaries of $f_C^0\cap f_D^0$ and of $f_C^1\cap f_D^1$ or the boundaries of $f_C^0\cap f_D^1$ and $f_C^1\cap f_D^0$.

In addition, we may choose $\widetilde{C}$ and $\widetilde{D}$ so that, if $\EF$ is a set of closed walks gene\-rating all closed walks of length smaller than~$|C|$, then $\EF$ generates $C$ or~$D$ as soon as it generates $\widetilde{C}$ or~$\widetilde{D}$.
\end{lem}

\begin{proof}
If $C$ and~$D$ are nested, then the assertion holds trivially.
This covers the situation that $n(D,C)$ is either~$0$ or~$1$, as $n(C,D)\in\{0,1\}$ implies that $C$ and~$D$ are nested.
In particular, we may assume that $C$ contains some smallest $D$-path $P_1$.
Note that the cases (\ref{itm_indecomp2}) and (\ref{itm_indecomp3}) of Lemma~\ref{lem_indecomp} imply $n(D,C)=1$.
Hence, Lemma~\ref{lem_indecomp} implies that $D$ contains a $C$-path $Q_1$ with the same end vertices as~$P_1$ and with $|P_1|=|Q_1|$.
By definition, neither $P_1$ nor $Q_1$ has an inner vertex that lies in~$D$ or~$C$, respectively.
Let $D':=D$ and let $C'$ be obtained from~$C$ by replacing $P_1$ with~$Q_1$.
Recursively, we obtain two sequences $(P_i)_{i\leq n}$ and $(Q_i)_{i\leq n}$ of $D$-paths in~$C$ and $C$-paths in~$D$, respectively, which are ordered by the length of the paths~$P_i$.
Note that~-- just as above~-- Lemma~\ref{lem_indecomp} ensures $|P_i|=|Q_i|$ for all but at most one $i\leq n$.
(The case with $|P_i|\neq|Q_i|$ occurs if $C$ and~$D$ are nested and either (\ref{itm_indecomp2}) or~(\ref{itm_indecomp3}) of Lemma~\ref{lem_indecomp} holds.)

Consider a cyclic ordering of~$C$ and let $i_1,\ldots,i_n\in\{1,\ldots,n\}$ be pairwise distinct such that $P_{i_1},\ldots,P_{i_n}$ appear on~$C$ in this order.
Then, using planarity, it immediately follows by their definitions as $C$-path or $D$-path, respectively, that $Q_{i_1},\ldots,Q_{i_n}$ appear in this order on~$D$.
Note that one face of $P_i\cup Q_i$ contains no vertices or edges of $C\cup D$.
The assertion follows except for the fact that the obtained closed walks are indecomposable and the additional statement.

Let $\EF$ be a set of closed walks generating all closed walks of length smaller than~$|C|$.
Assume that the boundaries $C'$ and $D'$ of $f_C^0\cap f_D^0$ and $f_C^1\cap f_D^1$, respectively, have the desired property up to being indecomposable.
Let us assume that $C'$ is generated by~$\EF$.
(Note that this covers also the case that $C'$ is not indecomposable.)

If all closed walks $P_i\inv Q_i$ have length less than $|C|$ and~$|D|$, then we add every closed walk $P_i\inv Q_i$ to~$C'$ -- after the canonical rotation -- for which $Q_i$ lies on the boundary of $f_C^0\cap f_D^0$ and we consider the smallest reduction. Thereby, we obtain~$C$.
So $C$ is generated by~$\EF$ as $C'$ and all of the added closed walks are generated by~$\EF$.

If all but exactly one of the closed walks $P_i\inv Q_i$ have length less than $|C|$ and~$|D|$, then $P_n\inv Q_n$ has largest length of all those closed walks.
If $P_n$ lies on the boundary of $f_C^0\cap f_D^0$, then we add every closed walk $P_i\inv Q_i$ to~$C'$ for which $Q_i$ lies on the boundary of $f_C^0\cap f_D^0$ and consider the smallest reduction.
As before, we obtain that $C$ is generated by~$\EF$.
If $P_n$ lies on the boundary of $f_C^1\cap f_D^1$, then we add every closed walk $P_i Q_i\inv$ to~$C'$ for which $P_i$ lies on the boundary of $f_C^0\cap f_D^0$ and obtain~$D$ and consider the smallest reduction.
So $D$ is generated by~$\EF$.

If at least two closed walks $P_i\cup Q_i$ have length at least $\min\{|C|,|D|\}$, then $n=2$ follows immediately.
Hence, the boundaries $C''$ and $D''$ of $f_C^0\cap f_D^1$ and $f_C^1\cap f_D^0$, respectively, are cycles.
So we may have chosen them instead of~$C'$ and~$D'$.
If one of them, $C''$ say, is generated by~$\EF$, too, then $C'(C'')\inv$ is generated by~$\EF$.
As this sum reduces to either~$C$ or~$D$, the assertion follows.
\end{proof}

Note that it follows from the proof of Lemma~\ref{lem_nestedCorners} that there is a canonical bijection between the $C$-paths in~$D$ and the $D$-paths in~$C$.
In particular, we have $n(C,D)=n(D,C)$.

\section{Counting crossing cycles}

Our restiction to finitely separable graphs implies that each cycle in such a planar graph is nested with all but finitely many cycles of bounded length, which directly carries over to indecomposable closed walks.\footnote{As cycles define closed walk canonically, nestedness of closed walks carries over to cycles in the obvious way. Equivalently, two cycles are \emph{nested} if neither has vertices or edges in both faces of the other and vice versa.}
Without finite separability this need not be true.

\begin{prop}\label{prop_MuFinite}
Let $i\in\nat$.
Every cycle in a finitely separable planar graph is nested with all but finitely many cycles of length at most~$i$.
\end{prop}

\begin{proof}
Let us assume that some cycle~$C$ is not nested with infinitely many cycles of length at most~$i$.
Then there are two vertices $x_1,x_2$ of~$C$ that lie on infinitely many of these cycles and thus we obtain infinitely many distinct $x_1$--$x_2$ paths of length at most $i-1$.
Either there are already infinitely many edge disjoint $x_1$--$x_2$ paths or infinitely many share another vertex~$x_3$. In the latter situation, there are either infinitely many distinct $x_1$--$x_3$ or $x_2$--$x_3$ paths of length at most $i-2$.
Continuing this process, we end up at some point with two distinct vertices and infinitely many edge disjoint paths between them, since we reduce the length of the involved paths in each step by at least~$1$.
So we obtain a contradiction to finite separability.
\end{proof}

Let $\EF$ be a set of indecomposable closed walks of length at most~$i$ in a finitely separable graph~$G$ and $C\sub G$ be an indecomposable closed walk.
We define $\mu_\EF(C)$ to be the number of elements of~$\EF$ that are not nested with~$C$.
Note that Proposition~\ref{prop_MuFinite} says that $\mu_\EF(C)$ is finite.
If $\FF$ is another set of indecomposable closed walks of length at most~$i$, we set $\mu_\EF(\FF)$ as minimum over all $\mu_\EF(C)$ with $C\in \FF$.

\begin{prop}\label{prop_muDecreases}
Let $G$ be a finitely separable planar graph.
Let $\EF$ be a set of indecomposable closed walks in~$G$ of length at most $i\in\nat$ and let $C,D$ be two indecomposable closed walks in~$G$ that are not nested.
Then we have
\[
\mu_\EF(C)+\mu_\EF(D)\geq\mu_\EF(\widetilde{C})+\mu_\EF(\widetilde{D}),
\]
where $\widetilde{C}$ and $\widetilde{D}$ are the closed walks obtained by Lemma~\ref{lem_nestedCorners}.
Furthermore, if $D\in\EF$, then the inequality is a strict inequality.
\end{prop}

\begin{proof}
Using homeomorphisms of the sphere, we may assume that $\widetilde{C}$ is the boundary of $f_C^0\cap f_D^0$ and $\widetilde{D}$ is the boundary of $f_C^1\cap f_D^1$.
Let $F\in\EF$ be nested with $C$ and~$D$.
We may assume that $F$ avoids $f_C^0$.
Thus, it is nested with $\widetilde{C}$.
If $F$ avoids $f_D^0$, too, then it lies in $f_C^1\cap f_D^1$ with its boundary and is nested with $\widetilde{D}$.
So let us assume that it avoids~$f_D^1$.
Thus, $F$ does not contain any points of $f_C^1\cap f_D^1$ and hence is nested with~$\widetilde{D}$.

Now consider the case that $F\in\EF$ is nested with~$C$ but not with~$D$.
We may assume that $F$ avoids~$f_C^0$.
Hence, it avoids $f_C^0\cap f_D^0$, too, and is nested with~$\widetilde{C}$.

This shows that every $F\in\EF$ that is not counted on the left side of the inequality is not counted on the right side either and that every $F\in\EF$ that is counted on the left side precisely once is counted on the right side at most once, which implies the first part of the assertion.

To see the additional statement, just note that $D$ is counted on the left for $\mu_\EF(C)$ but not for~$\mu_\EF(D)$ and that both closed walks $\widetilde{C}$ and $\widetilde{D}$ are nested with~$D$.
\end{proof}

\section{Finding a nested generating set}

The main theorem of~\cite{H-GeneratingCycleSpace} says that the cycle space of any $3$-connected finitely separable planar graph~$G$ is generated by some canonical nested set of cycles as $\mathbb F_2$-vector space.
We shall prove the analogous result for the set $\WF(G)$ of all closed walks.

Throughout this section, let $G$ be a $3$-connected planar finitely separable graph.
Let $\WF_i:=\WF_i(G)$ be the subset of~$\WF(G)$ generated by all closed walks of length at most~$i$.
So $\WF(G)=\bigcup_{i\in\nat}\WF_i$.
We shall recursively define canonical nested subsets $\CF_i$ of~$\WF_i$ that generate~$\WF_i$ and consist only of indecomposable closed walks of length at most~$i$.
So $\bigcup_{i\in\nat}\CF_i$ will generate~$\WF(G)$.
We shall define the $\CF_i$ recursively.
For the start, let $\CF_i=\es$ for $i\leq 2$.
Now let us assume that we already defined $\CF_{i-1}$.

In order to define $\CF_i$, we construct another sequence of nested $\Aut(G)$-invariant sets $\CF_i^\kappa$ of indecomposable closed walks.
Set $\CF_i^0:=\CF_{i-1}$.
Let $\kappa$ be some ordinal such that $\CF_i^\lambda$ is defined for all $\lambda<\kappa$.
If $\kappa$ is a limit ordinal, then set $\CF_i^\kappa=\bigcup_{\lambda<\kappa}\CF_i^\lambda$.
So let $\kappa$ be a successor ordinal, say $\kappa=\nu+1$.
Any closed walk of length~$i$ that is not generated by~$\CF_i^\nu$ must be indecomposable by definition of~$\CF_{i-1}$.
If there is not such a closed walk, set $\CF_i:=\CF_i^\nu$.
So in the following, we assume that there is at least one indecomposable closed walk of length~$i$ that is not generated by $\CF_i^\nu$.
Hence, the set $\DF_i^\kappa$ of all indecomposable closed walks of length~$i$ that are not generated by~$\CF_i^\nu$ is not empty.

\begin{lem}\label{lem_DHasANiceCycle}
The set $\DF_i^\kappa\neq\es$ contains a closed walk that is nested with~$\CF_i^\nu$.
\end{lem}

\begin{proof}
Let $C\in \DF_i^\kappa$ with minimum $\mu_{\CF_i^\nu}(C)$.
(As all involved closed walks are indecomposable, $\mu_{\CF_i^\nu}(C)$ is well-defined.)
We shall show $\mu_{\CF_i^\nu}(C)=0$.
So let us suppose that $C$ is not nested with some $D\in\CF_i^\nu$.
Since $C$ and~$D$ are indecomposable, we obtain by Lemma~\ref{lem_nestedCorners} two indecomposable closed walks $\widetilde{C}$ and $\widetilde{D}$ with $|C|=|\widetilde{C}|$ and $|D|=|\widetilde{D}|$ such that Proposition~\ref{prop_muDecreases} implies
\[
\mu_{\CF_i^\nu}(C)=\mu_{\CF_i^\nu}(C)+\mu_{\CF_i^\nu}(D)>\mu_{\CF_i^\nu}(\widetilde{C})+\mu_{\CF_i^\nu}(\widetilde{D}).
\]
Note that, if $\widetilde{C}$ and $\widetilde{D}$ are generated by~$\CF_i^\nu$, then $C$ being generated by~$D$, $\widetilde{C}$, and $\widetilde{D}$ is generated by~$\CF_i^\nu$, too.
But then it does not lie in~$\DF_i^\kappa$.
As it does, either $\widetilde{C}$ or~$\widetilde{D}$ is not generated by~$\CF_i^\nu$.
In particular, this closed walk must lie in $\DF_i^\kappa$, a contradiction to the choice of~$C$.
\end{proof}

Let $\EF_i^\kappa$ be the set of all closed walks in~$\DF_i^\kappa$ that are nested with~$\CF_i^\nu$.
By Lemma~\ref{lem_DHasANiceCycle}, this set is not empty.

For a set $\EF$ of closed walks of length at most~$i$, we call $C\in\EF$ \emph{optimally nested} in~$\EF$ if $\mu_\EF(C)=\mu_\EF(\EF)$.
Note that $\mu_\EF(\EF)$ is finite by Proposition~\ref{prop_MuFinite} and, furthermore, as $3$-connected planar graphs have (up to homeomorphisms) unique embeddings into the sphere due to Whitney~\cite{whitney_congruent_1932} for finite graphs and Imrich~\cite{I-Whitney} for infinite graphs, $\mu_\EF(C)=\mu_\EF(C\alpha)$ for all $\alpha\in\Aut(G)$.

\begin{lem}\label{lem_OptNestedAreAutInvariant}
The set $\FF_i^\kappa$ of optimally nested closed walks in $\EF_i^\kappa$ is non-empty and nested.
\end{lem}

\begin{proof}
Since $\EF_i^\kappa$ is non-empty, the same is true for~$\FF_i^\kappa$.
Let us suppose that $\FF_i^\kappa$ contains two closed walks $C,D$ that are not nested.
Let $\widetilde{C}$ and $\widetilde{D}$ be the indecomposable closed walks obtained by Lemma~\ref{lem_nestedCorners} with $|C|=|\widetilde{C}|$ and $|D|=|\widetilde{D}|$ each of which is not generated by~$\CF_i^\nu$ and such that Proposition~\ref{prop_muDecreases} yields
\[
\mu_{\CF_i^\nu}(C)+\mu_{\CF_i^\nu}(D)\geq\mu_{\CF_i^\nu}(\widetilde{C})+\mu_{\CF_i^\nu}(\widetilde{D}).
\]
As $\EF_i^\kappa$ is nested with~$\CF_i^\nu$ by definition, we have $\mu_{\CF_i^\nu}(C)+\mu_{\CF_i^\nu}(D)=0$.
Note that $\widetilde{C}$ and~$\widetilde{D}$ lie in~$\DF_i^\kappa$ by definition.
As both are nested with $\CF_i^\nu$, they lie in~$\EF_i^\kappa$.
We apply Proposition~\ref{prop_muDecreases} once more and obtain
\[
\mu_{\EF_i^\kappa}(C)+\mu_{\EF_i^\kappa}(D)>\mu_{\EF_i^\kappa}(\widetilde{C})+\mu_{\EF_i^\kappa}(\widetilde{D}).
\]
Thus either $\widetilde{C}$ or $\widetilde{D}$ is not nested with less elements of~$\EF_i^\kappa$ than~$C$.
This contradiction to the choice of~$C$ shows that $\EF_i^\kappa$ is nested.
\end{proof}

So we set $\CF_i^\kappa:=\CF_i^\nu\cup\FF_i^\kappa$.
Then $\CF_i^\kappa$ is nested as $\CF_i^\nu$ is nested and by the choice of~$\EF_i^\kappa$ all elements of $\CF_i^\kappa$ are indecomposable.

This process will terminate at some point as we strictly enlarge the sets $\CF_\kappa$ in each step but we cannot put in more closed walks than there are in~$G$.
Let $\CF_i$ be the union of all $\CF_i^\kappa$.
Note that we made no choices at any point, i.\,e.\ all sets $\CF_i$ are $\Aut(G)$-invariant and canonical.
Thus, we proved Theorem~\ref{thm_canGen}.
More precisely, we have proved the following theorem.

\begin{thm}\label{thm_mainCWNested}
For every finitely separable $3$-connected planar graph~$G$ there is a sequence $(\CF_i)_{i\in\nat}$ of sets of closed walks in~$G$ such that
\begin{enumerate}[\rm (i)]
\item $\CF_{i-1}\sub\CF_i$;
\item $\CF_i\sm\CF_{i-1}$ consists of indecomposable closed walks of length~$i$;
\item $\CF_i$ generates $\WF_i(G)$;
\item $\CF_i$ is canonical and nested.
\end{enumerate}

In particular, $\WF(G)$ has a canonical nested generating set.\qed
\end{thm}

Note that the only situation where we used $3$-connectivity was when we concluded that we have $\mu_\EF(C)=\mu_\EF(C\alpha)$ for any closed walk~$C$, set~$\EF$ of closed walks of bounded length and automorphism~$\alpha$.
That is, the above proof also give us the existence of a nested generating set for lower connectivity, but we lose canonicity.
Note that, in general, not only our proof fails but the statement of Theorem~\ref{thm_canGen} is false if we do not require the graph to be $3$-connected:
let $G$ be the graph obtained by two vertices joined by four internally disjoint paths of length~$2$.
Then all cycles have length~$4$ and lie in the same $\Aut(G)$-orbit, but it is not hard to find two of them which are not nested.
So you cannot find a canonical nested generating set of~$\WF(G)$ consisting only of indecomposable closed walks.
Similarly, whichever generating set you take, none of its elements is nested with all of its $\Aut(G)$-images.

\section{Finding a finite generating set}\label{sec_planar}

We call a graph \emph{quasi-transitive} if its automorphism group has only finitely many orbits on the vertex set.
If a group $\Gamma$ acts on a graph~$G$, we denote by $|G/\Gamma|$ the number of $\Gamma$-orbits on~$G$.
In particular, if~$G$ is quasi-transitive, then $|G/\Aut(G)|<\infty$.
If $H$ is a subgraph of~$G$, we denote by $\Gamma_H$ the (setwise) stabiliser of~$H$ in~$\Gamma$.

In this section, we give up nestedness of our generating set for $\WF(G)$ in order to obtain a generating set consisting of only finitely many orbits.
More precisely, we shall prove the following theorem.

\begin{thm}\label{thm_mainCW}
Let $G$ be a locally finite quasi-transitive planar graph.
Then $\WF(G)$ has an $\Aut(G)$-invariant generating set that consists of finitely many orbits.
\end{thm}

Let us introduce the notion of a degree sequence of orbits because the general idea to prove Theorem~\ref{thm_main} will mainly be done by induction on this notion.

Let $\Gamma$ act on a locally finite graph $G$ with $|V(G)|>1$ such that $|G/\Gamma|$ is finite.
We call a tupel $(d_1,\ldots, d_m)$ of positive integers with $d_i\geq d_{i+1}$ for all $i<m$ the \emph{degree sequence of the orbits of~$(G,\Gamma)$} if for some set $\{v_1,\ldots, v_m\}$ of vertices that contains precisely one vertex from each $\Gamma$-orbit the degree of~$v_i$ is~$d_i$.
We consider the lexicographic order on the finite tupels of positive integers (and thus on the degree sequences of orbits), that is, we set
\[
(d_1,\ldots, d_m)\leq (c_1,\ldots, c_n)
\]
if either $m\leq n$ and $d_i=c_i$ for all $i\leq m$ or $d_i<c_i$ for the smallest $i\leq m$ with $d_i\neq c_i$.
Note that any two finite tupels of positive integers are $\leq$-comparable.

A direct consequence of this definition is the following lemma.

\begin{lem}\label{lem_IndOnDegSequPre}
Any strictly decreasing sequence in the set of finite tupels of positive integers is finite.\qed
\end{lem}

Lemma~\ref{lem_IndOnDegSequPre} for degree sequences of orbits reads as follows and enables us to use induction on the degree sequence of the orbits of graphs:

\begin{lem}\label{lem_IndOnDegSequ}
Let $(G_i,\Gamma_i)$ be a sequence of pairs, where $G_i$ is a locally finite graph and $\Gamma_i$ acts on~$G_i$ such that $|G_i/\Gamma_i|$ is finite.
If the corresponding sequence of degree sequences of the orbits is strictly decreasing, then the sequence $(G_i,\Gamma_i)$ is finite.\qed
\end{lem}

\begin{lem}\label{lem_SmallerDegSequ}
Let $G$ be a locally finite graph and let $Gamma$ act on~$G$ such that $|G/\Gamma|$ is finite.
Let $S\sub V(G)$ and $H\sub G$ be such that the following conditions hold:
\begin{enumerate}[\rm (i)]
\item\label{itm_SmallerDegSequ_1} $G-S$ is disconnected;
\item\label{itm_SmallerDegSequ_2} each $S\alpha$ with $\alpha\in\Gamma$ meets at most one component of $G-S$;
\item\label{itm_SmallerDegSequ_3} such that no vertex of~$S$ has all its neighbours in~$S$;
\item\label{itm_SmallerDegSequ_4} $H$ is a maximal subgraph of~$G$ such that no $S\alpha$ with $\alpha\in\Gamma$ disconnects~$H$.
\end{enumerate}
Then the degree sequence of the orbits of~$(H,\Gamma_H)$ is smaller than the one of~$(G,\Gamma)$.
\end{lem}

\begin{proof}
First we show that all vertices in~$H$ that lie in a common $\Gamma$-orbit of~$G$ and whose degrees in~$G$ and in~$H$ are the same also lie in a common $\Gamma_H$-orbit.
Let $x,y$ be two such vertices and $\alpha\in\Gamma$ with $x\alpha=y$.
Suppose that $H\alpha\neq H$.
Then there is some $S\beta$ that separates some vertex of~$H$ from some vertex of~$H\alpha$ by the maximality of~$H$.
But as $y$ and all its neighbours lie in $H$ and in $H\alpha$, they lie in $S\beta$, which is a contradiction to~(\ref{itm_SmallerDegSequ_3}).
Thus, we have $\alpha\in\Gamma_H$.

Now, we consider vertices $x$ such that $\{x\}\cup N(x)$ lies in no $H\alpha$ with $\alpha\in\Gamma$ and such that $x$ has maximum degree with this property.
Let $\{x_1,\ldots, x_m\}$ be a maximal set that contains precisely one vertex from each orbit of those vertices.
If $x_i$ lies outside every $H\alpha$, then no vertex of its orbit is considered for the degree sequence of the orbits of~$(H,\Gamma_H)$.
If $x_i$ lies in~$H$, then its degree in some $H\alpha$ is smaller than its degree in~$G$.
By replacing $x_i$ by $x_i\alpha\inv$, if necessary, we may assume $d_H(x_i)<d_G(x_i)$.
So its value in the degree sequence of orbits of~$(H,\Gamma_H)$ is smaller than its value in the degree sequence of orbits of~$(G,\Gamma)$; but it may be counted multiple times now as the $\Gamma$-orbit containing $x_i$ may be splitted into multiple $\Gamma_H$-orbits.
Nevertheless, the degree sequence of orbits of~$(H,\Gamma_H)$ is smaller than that of~$(G,\Gamma)$.
\end{proof}

Remember that a \emph{block} of a graph is a maximal $2$-connected subgraph.
As any indecomposable closed walk is a cycle and hence lies completely in some block and as any locally finite quasi-transitive graph has only finitely many orbits of blocks, we directly have:

\begin{prop}\label{prop_red1to2}
Let $G$ be a locally finite quasi-transitive graph and let $\Gamma$ act on~$G$ such that $|G/\Gamma|$ is finite.
Then $\WF(G)$ has a $\Gamma$-invariant generating set consisting of finitely many orbits if and only if the same is true for every block $B$ with respect to the action of~$\Gamma_B$.\qed
\end{prop}

\begin{rem}
In the situation of Proposition~\ref{prop_red1to2} we can take the orbits of the cutvertices one-by-one and apply Lemma~\ref{lem_SmallerDegSequ} for each such orbit.
It follows recursively that each block has a smaller degree sequence of its orbits than the original graph.
Since $|G/\Gamma|$ is finite, there are only finitely many orbits of cut vertices.
So we stop at some point.
\end{rem}

For the reduction to the $3$-connected case for graphs of connectivity~$2$, we apply Tutte's decomposition of $2$-connected graphs into `$3$-connected parts' and cycles.
Tutte~\cite{Tutte} proved it for finite graphs.
Later, it was extended by Droms et al.~\cite{DSS-Tutte} to locally finite graphs.

A \emph{\td} of a graph $G$ is a pair $(T,\cV)$ of a tree $T$ and a family $\cV=(V_t)_{t\in T}$ of vertex sets $V_t\sub V(G)$, one for each vertex of~$T$, such that
\begin{enumerate}[(T1)]
\item $V = \bigcup_{t\in T}V_t$;
\item for every edge $e\in G$ there exists a $t \in V(T)$ such that both ends of $e$ lie in~$V_t$;
\item $V_{t_1} \cap V_{t_3} \sub V_{t_2}$ whenever $t_2$ lies on the $t_1$--$t_3$ path in~$T$.
\end{enumerate}

The sets $V_t$ are the \emph{parts} of $(T,\cV)$ and the intersections $V_{t_1}\cap V_{t_2}$ for edges $t_1t_2$ of~$T$ are its \emph{adhesion sets}; the maximum size of such a set is the \emph{adhesion} of $(T,\cV)$.
Given a part $V_t$, its \emph{torso} is the graph with vertex set $V_t$ and whose edge set is
\[
\{xy\in E(G)\mid x,y\in V_t\}\ \cup\ \{xy\mid \{x,y\}\sub V_t\text{ lies in an adhesion set}\}.
\]

If $\Gamma$ acts on~$G$, then it acts canonically on vertex sets of~$G$.
If every part of the \td\ is mapped to another of its parts and this map induces an automorphism of~$T$ then we call the \td\ \emph{$\Gamma$-invariant}.

\begin{thm}\cite[Theorem 1]{DSS-Tutte}\label{thm_tutte}
Every locally finite $2$-connected graph $G$ has an $\Aut(G)$-invariant \td\ of adhesion~$2$ each of whose torsos is either $3$-connected or a cycle or a complete graph on two vertices.
\qed
\end{thm}

\begin{rem}
In addition to the conclusion of Theorem~\ref{thm_tutte}, we may assume that the \td\ is such that the torsos of tree vertices of degree~$2$ are either $3$-connected or cycles and that no two torsos of adjacent tree vertices $t_1,t_2$ are cycles if $V_{t_1}\cap V_{t_2}$ is no edge of~$G$.
(Remember that edges are two-element vertex sets.)
We call a \td\ as Theorem~\ref{thm_tutte} with this additional property a \emph{Tutte decomposition}.
\end{rem}

Now we reduce the problem of Theorem~\ref{thm_mainCW} from $2$-connected graphs to $3$-connec\-ted ones.

\begin{prop}\label{prop_red2to3}
Let $G$ be a locally finite $2$-connected graph and let $\Gamma$ act on~$G$ such that $|G/\Gamma|$ is finite.
Then $\WF(G)$ has a $\Gamma$-invariant generating set consisting of finitely many orbits if and only if the same is true for each of its torsos $B$ in every Tutte decomposition with respect to the action of~$\Gamma_B$.
\end{prop}

\begin{proof}
Let $(T,\cV)$ be a Tutte decomposition of~$G$.
Note that every vertex lies in only finitely many $2$-separators (cf.~\cite[Proposition 4.2]{ThomassenWoess}).
Thus, the graph $H$ given by~$G$ together with all edges $xy$, where $\{x,y\}$ forms an adhesion set, is also locally finite, the action of~$\Gamma$ on~$G$ extends canonically to an action on~$H$ and we have $|H/\Gamma|<\infty$ for this action.
There are only finitely many orbits of (the action induced by) $\Aut(G)$ on~$T$, since any $2$-separator of~$G$ uniquely determines the parts $V_t$ of~$(T,\cV)$ it is contained in and since there are only finitely many $\Aut(G)$-orbits of $2$-separators.
Obviously, the restriction of $H$ to any $V_t\in\VF$ is the torso of $V_t$.

\medskip

Let us assume that $\WF(G)$ has a $\Gamma$-invariant generating set consisting of finitely many orbits and let $\CF$ be a finite set of closed walks that generates together with its images $\WF(G)$.
Every $C\in\CF$ can be generated by (finitely many) indecomposable closed walks $C_1,\ldots, C_n$ in~$H$.
So the set~$\DF$ of all those~$C_i$ for all $C\in\CF$ together with the images under $\Gamma$ generates $\WF(H)$.
Each of the closed walks $C_i$ lies in a unique part $V_t$ of~$(T,\VF)$ as they have no shortcut and as every adhesion set in~$H$ is complete.
Note that closed walks which lie in the same $\Gamma$-orbit and in some $V_t$ also lie in the same orbit with respect to the automorphisms of the torso $G_t$ of~$V_t$.
Let $\DF_t$ be the set of all closed walks in~$\DF$ that lie in~$G_t$.
Let $C$ be a closed walk in~$G_t$.
Then it is generated by $C_1,\ldots,C_n\in\DF$.
Since all $C_i\not\sub G_t$ add to spikes, those $C_i\sub G_t$ cancel out.
Thus, $\WF(G_t)$ has an $\Gamma_{G_t}$-invariant set of closed walks consisting of finitely many $\Gamma_{G_t}$-orbits.

\medskip

For the converse, let $\WF(G_t)$ for every torso $G_t$ of~$(T,\VF)$ have a $\Gamma_{G_t}$-invariant generating set~$\CF_t$ of closed walks consisting of finitely many $\Gamma_{G_t}$-orbits.
We may choose the sets $\CF_t$ so that $\CF_t=\CF_{t'}\alpha$ if $\alpha\in\Gamma$ maps~$V_t$ to~$V_{t'}$.
Let $\AF$ be a set of ordered adhesion sets $(x,y)$ of~$(T,\VF)$ consisting of one element for each $\Gamma$-orbit.
For every $(x,y)\in\AF$ with $xy\notin E(G)$ we fix an $x$--$y$ path $P_{xy}$ in~$G$.
Then $xP_{xy}yx$ is a closed walk $C_{xy}$ in~$H$.
If $xy\in E(G)$, let $P_{xy}=xy$ and, for later conveniences, let $C_{xy}=\es$ be the empty walk.
Note that for an adhesion set $\{x,y\}$ we may have fixed two distinct paths $P_{xy}$ and~$P_{yx}$.
We canonically extend the definition of the paths $P_{xy}$ and cycles $C_{xy}$ to all ordered adhesion sets $(x,y)$, i.e.\ if $(x,y)=(x',y')\alpha$ with $(x',y')\in\AF$, set $P_{xy}:=P_{(x'y')}\alpha$ and $C_{xy}:=C_{(x'y')}\alpha$.

Note that there are only finitely many $\Gamma$-orbits of parts of~$(T,\VF)$.
So the union $\CF$ of all $\CF_t$ is a set of closed walks in~$H$ meeting only finitely many $\Gamma$-orbits and generating $\WF(H)$, as it has a generating set of induced closed walks, each of those lies in some~$G_t$ and thus is generated by~$\CF$.
For every $C\in\CF$ let $W_C$ be the element of~$\WF(G)$ that is obtained from~$C$ by replacing its edges $xy$ that form an adhesion set $\{x,y\}$ of~$(T,\VF)$ by~$P_{xy}$.
Let $\CF':=\{W_C\mid C\in\CF\}$.

To see that $\CF'$ generates $\WF(G)$, let $C$ be any closed walk of~$G$.
Thus it is also a closed walk of~$H$ and is generated by some $C_1,\ldots,C_m\in\CF$.
Now we replace each edge $xy$~-- passed in this order on the walk~-- on any of these $C_i$ that forms an adhesion set of $(T,\VF)$ by its path $P_{xy}$ and obtain a closed walk $C_i'$. (Formally, we insert the closed walk $yxP_{xy}$ directly after passing $xy$ and remove the spike $xyx$.)
Then $C_i'$ lies in~$\WF(G)$ since it contains no edge of $H\sm G$.
We now follow the sums, reductions and rotations and addings of spinkes we used to generate $C$ from the $C_i$. Each time we removed a spike $xyx$ for an adhesion set $\{x,y\}$ of $(T,\VF)$, we instead remove many spike, namely $P_{xy}P_{xy}\inv$.
In that way, the $C_i'$ generate $C$, too.
Thus $C$ can be generated by~$\CF'$.
\end{proof}

\begin{rem}\label{rem_TutteForInduc}
Unfortunately, we are not able to apply Lemma~\ref{lem_SmallerDegSequ} directly for Proposition~\ref{prop_red2to3} to see that the torsos in a Tutte decomposition have a smaller degree sequence of orbits, as the orbits are not subgraphs of~$G$.
But as not both vertices of any adhesion set have degree~$2$, it is possible to follow the argument of the proof of Lemma~\ref{lem_SmallerDegSequ} for each of the finitely many orbits of the $2$-separators one-by-one to see that each torso has a smaller degree sequence of orbits than~$G$.
\end{rem}

Now we are able to attack the general VAP-free case.

\begin{prop}\label{prop_VAPFreeAccess}
Let $G$ be a locally finite VAP-free planar graph and let $\Gamma$ act on~$G$ such that $|G/\Gamma|$ is finite.
Then $\WF(G)$ has an $\Aut(G)$-invariant generating set consisting of only finitely many orbits.
\end{prop}

\begin{proof}
Due to Propositions~\ref{prop_red1to2} and~\ref{prop_red2to3}, it suffices to show the assertion if~$G$ is $3$-connected.
As $3$-connected planar graphs have (up to homeomorphisms) unique embeddings into the sphere, every automorphism of~$G$ induces a homeomorphism of the plane.
So faces are mapped to faces and closed walks that are face boundaries are mapped to such walks.
As $G$ is locally finite and $|G/\Gamma|<\infty$, there are only finitely many $\Gamma$-orbits of finite face boundaries.

Since $\WF(G)$ is generated by the indecomposable closed walks, it suffices to prove that every indecomposable closed walk is generated by the face boundaries.
Since every indecomposable closed walk $W$ is a cycle in~$G$, it determines an inner face and an outer face in the plane.
The inner face of~$W$ contains only finitely many edges as $G$ is VAP-free.
Let $xy$ be an edge of~$W$ and $f$ the face of~$G$ in the inner face of~$W$ containing~$e$.
Let $P_{xy}$ be the second $x$-$y$ path apart from $xy$ on the boundary of~$f$.
Replacing in~$W$ the edge $xy$ by $P_{xy}$ is summing $yxP_{xy}$ to~$W$ and removing the spike $xyx$.
Thus, the resulting closed walk $W'$ is generated by the face boundaries if and only if $W$ is generated by them.
Inductively on the number of edges in the inner face of~ $W'$, we obtain the assertion.
\end{proof}

Now we are able to prove that $\WF(G)$ has a finite generating set as $\Gamma$-module.

\begin{thm}\label{thm_PlanarFiniteModule}
Let $G$ be a locally finite planar graph and let $\Gamma$ act on~$G$ such that $|G/\Gamma|$ is finite.
Then $\WF(G)$ has an $\Gamma$-invariant generating set consisting of only finitely many orbits.
\end{thm}

\begin{proof}
Due to Propositions~\ref{prop_red1to2} and~\ref{prop_red2to3}, we may assume that $G$ is $3$-connected and due to Proposition~\ref{prop_VAPFreeAccess} we may assume that $G$ is not VAP-free.
Let $\varphi\colon G\to\real^2$ be a planar embedding of~$G$.
Let $\CF$ be a non-empty $\Gamma$-invariant nested set of indecomposable closed walks that generates $\WF(G)$, which exists by Theorem~\ref{thm_mainCWNested}.
Since $G$ is not VAP-free, there is some cycle $C$ of~$G$ such that both faces of $\real^2\sm \varphi(C)$ contain infinitely many vertices of~$G$.
As $\CF$ generates $\WF(G)$, one of the indecomposable closed walks in~$\CF$ has the same property as~$C$.
Hence, we may assume $C\in\CF$.
In particular, $\{C\alpha\mid\alpha\in\Gamma\}$ is nested.

We consider maximal subgraphs $H$ of~$G$ such that no $C\alpha$ with $\alpha\in\Gamma$ disconnects~$H$.
In particular, $H$ is connected and for every $C\alpha$ with $\alpha\in\Gamma$ one of the faces of ${\real^2\sm \varphi(C\alpha)}$ is disjoint from~$H$.
Note that there are only finitely many $\Gamma$-orbits of such subgraphs~$H$ as we find in each orbit some element that contains vertices of~$C$ by maximality of~$H$.
Due to Lemma~\ref{lem_SmallerDegSequ}, the pair~$(H,\Gamma_H)$ has a strictly smaller degree sequence of its orbits than $(G,\Gamma)$ as $C$ disconnects~$G$.
Since $H$ is again a locally finite planar graph and $|G/\Gamma|<\infty$, we conclude by induction on the degree sequence of the orbits of such graphs (cf.\ Lemma~\ref{lem_IndOnDegSequ}) with base case if~$G$ is VAP-free that $\WF(H)$ has a $\Gamma_H$-invariant generating set consisting of finitely many $\Gamma_H$-orbits.
Let $\EF_H$ be such a set.

There are only finitely many pairwise non-$\Gamma$-equivalent such subgraphs~$H$.
So let $\HF$ be a finite set of such subgraphs consisting of one per $\Gamma$-orbit.
Let 
\[
\EF:=\bigcup_{H\in\HF}\bigcup_{\alpha\in\Gamma}\EF_H\alpha.
\]
Then $\EF$ is $\Gamma$-invariant and has only finitely many orbits.
We shall show that $\EF$ generates~$\WF(G)$.
It suffices to show that every indecomposable closed walk is generated by~$\EF$.

Let $D$ be an indecomposable closed walk of~$G$.
If $D$ lies entirely inside some of the subgraphs~$H\in\HF$ or its $\Gamma$-images, then, obviously, it is generated by~$\EF$.
So let us  assume that there is some $\alpha\in\Gamma$ such that both faces of $C\alpha$ contain vertices or edges of~$D$.
By considering $D\alpha\inv$ instead of~$D$, we may assume $\alpha=1_\Gamma$.
We add all vertices and edges of~$C$ to~$D$ that lie in the bounded face of~$D$ to obtain a subgraph~$F$ of~$G$.
Then $D$ is the generated by all boundaries $C_1,\ldots, C_k$ of bounded faces of~$F$.

Assume that $C\beta$ with $\beta\in\Gamma$ is not nested with~$C_i$ and suppose that it is nested with~$D$.
Remember that $C$ and~$C\beta$ are nested.
Since $C\beta$ contains points in both faces of~$C_i$, there is some (possibly trivial) common walk $P$ of~$C_i$ and $C\beta$ such that the edges on $C\beta$ incident with the end vertices of~$P$ lie in different faces of~$C_i$ and also the edges of~$C_i$ incident with the end vertices of~$P$ lie in different faces of~$C\beta$.
As $C\beta$ is nested with $C$ and with~$D$, one of these edges belongs to~$C$ and the other to~$D$.
Thus, $C$ and~$D$ must lie in distinct faces of~$C\beta$ and hence must be nested.
This contradiction shows that every $C\beta$ that is not nested with~$C_i$ is not nested with~$D$ either.

As $C$ is not nested with~$D$ but with every~$C_i$, every $C_i$ is not nested with less closed walks $C\beta$ than $D$ and this is a finite number by Proposition~\ref{prop_MuFinite} as all involved closed walks are indecomposable and all closed walks $C\beta$ have the same length.
Induction on the number of closed walks $C\beta$ the current closed walk is not nested with implies that each $C_i$ is generated by~$\EF$ and so is~$D$.
\end{proof}

\section{Fundamental group of planar graphs}

In this section, we want to find two special generating sets for the fundamental group of planar graphs~$G$.
In order to do that, we first prove a general statement about the interplay of generating sets for $\WF(G)$ and for $\pi_1(G)$.
Note that, if $W$ is a closed walk starting and ending at a vertex~$v$, we denote by $[W]$ the homotopy class of~$W$.

\begin{prop}\label{prop_GenPi}
Let $G$ be a planar graph, let $v\in V(G)$, and let $\VF$ be a generating set for $\WF(G)$ that is closed under taking inverses.
Then
\[
\VF_\pi:=\{[P_WWP_W\inv]\mid W\in\VF, P_W \text{ is a $v$-$W$ walk}\}
\]
generates $\pi_1(G)$.
\end{prop}

\begin{proof}
Let $\eta\in\pi_1(G)$ and $W\in\eta$ be a reduced closed walk.
Then $W$ is generated by $W_1,\ldots, W_\ell\in\VF$.
We assume that the walks $W_i$ were used in this order to generate~$W$, in particular, there is a closed walk $R$ that starts at~$v$ and is generated by $W_1,\ldots W_{\ell-1}$ such that $R$ and $W_\ell$ generate~$W$.
By induction on~$\ell$, we may assume that $[R]\in\pi_1(G)$.

Since $R$ and $W_\ell$ generate $W$, there is some vertex $x_0$ on~$R$ such that adding spikes recursively, that is, adding a `large' spike $x_0x_1\ldots x_nx_{n-1}\ldots x_0$, and then inserting a rotation of~$W_\ell$ at $x_n$ results in~$W$.
(Note that we can assume that we need not take the inverse of~$W_\ell$ since $\VF$ is closed under taking inverses.)
But then $W$ is just the same as $PR$ for $P:=vRx_0\ldots x_nW_\ell x_n\ldots x_0R\inv v$.
Since $[R]$ is already generated and $[P]\in\VF_\pi$, we conclude that $[W]$ is generated by~$\VF_\pi$.
\end{proof}

For any $\eta\in\pi_1(G)$, let $P_\eta\in\eta$ be the unique reduced closed walk in~$\eta$ and $P_\eta^\circ$ be its cyclical reduction.
Similarly to the proof of the uniqueness of~$P_\eta$, it is possible to show that $P_\eta^\circ$ is unique.
If $\VF_\pi\sub\pi_1(G)$, set
\[
\VF_\pi^\circ:=\{P_\eta^\circ\mid\eta\in\VF_\pi\}.
\]

Now we are able to prove that the fundamental group of every planar finitely separable graph has a canonical generating set that comes from a nested generating set of~$\WF(G)$.

\begin{thm}
Let $G$ be a planar $3$-connected finitely separable graph.
Then $\pi_1(G)$ has a generating set $\VF_\pi$ such that $\VF_\pi^\circ$ is a canonical nested generating set for $\WF(G)$ consisting only of indecomposable closed walks.
\end{thm}

\begin{proof}
Let $v\in V(G)$ and $\VF$ be a canonical nested set of closed walks generating $\WF(G)$ such that $\VF$ consists of indecomposable closed walks.
This set exists by Theorem~\ref{thm_mainCWNested}.
Then the set
\[
\VF_\pi:=\{[P_WWP_W\inv]\mid W\in\VF, P_W \text{ is a $v$-$W$ walk}\}
\]
generates $\pi_1(G)$ by Proposition~\ref{prop_GenPi}.
Since $\VF=\VF_\pi^\circ$, the assertion follows.\end{proof}

In the second theorem on the fundamental group, we look at the sitution in Cayley graphs $G$ of finitely generated groups~$\Gamma$ and for a generating set of $\pi_1(G)$ consisting of only finitely many orbits.
But in order to talk about orbits, we have to define the action on the fundamental group.
If $\Gamma=\langle S\mid R\rangle$, let $\FF_S$ be the free group freely generated by~$S$.
For a word $w\in\FF_S$, let $P_w$ be the walk in~$G$ that starts at the vertex~$v$ and corresponds to the word~$w$, where $v$ is the vertex representing the group element $1_\Gamma$.
We assume that $\pi_1(G)$ is defined with respect to the base vertex~$v$.
Let $W$ be a closed walk in~$G$ that starts at~$v$ and let $W_w$ be the image of~$W$ under the action of the element $g_w\in\Gamma$ that is given by~$w$.
Then $P_w(W_w)P_w\inv$ is a closed walk with first vertex~$v$, it is the image of~$W$ under~$w$.
In this way, $\FF_S$ acts on the closed walks starting at~$v$, and as the images of homotopy equivalent closed walks are again homotopy equivalent, $\FF_S$ acts on $\pi_1(G)$.

\begin{thm}\label{thm_PiFinGen}
Let $G$ be a locally finite planar Cayley graph of a finitely generated planar group~$\Gamma=\langle S\mid R\rangle$.
Then $\pi_1(G)$ has a generating set consisting of finitely many $\FF_S$-orbits.
\end{thm}

\begin{proof}
Let $v$ be the vertex of~$G$ corresponding to $1\in\Gamma$ and let $\VF$ be a generating set of $\WF(G)$ consisting of only finitely many $\FF_S$-orbits.
This exists by Theorem~\ref{thm_PlanarFiniteModule}.
By Proposition~\ref{prop_GenPi}, it suffices to show that the set
\[
\VF_\pi:=\{[P_WWP_W\inv]\mid W\in\VF, P_W \text{ is a $v$-$W$ walk}\}
\]
has only finitely many $\FF_S$-orbits.
To see this, it suffices to show that any two $[P_WWP_W\inv]$ and $[Q_WWQ_W\inv]$, where $P_W$ and $Q_W$ are $v$-$W$ walks, are in the same $\FF_S$-orbit.
But this is immediate: just take the group element corresponding to the word $w$ defined by the walk $P_WxWyQ_W\inv$, where $x$ is the end vertex of~$P_w$ and $y$ is the end vertex of~$Q_W$.
Since conjugation of $[P_WWP_W\inv]$ by~$W$ is $[Q_WWQ_W\inv]$, the assertion follows.
\end{proof}

Theorem~\ref{thm_PiFinGen} has an immediate consequence to groups:
Droms~\cite{D-InfiniteEndedGroups} proved that finitely generated planar groups are finitely presented.
His proof uses an accessibility result of Maskit~\cite{Maskit65}.
As an application of Theorem~\ref{thm_PiFinGen} we obtain a self-contained proof of Droms's result.

Let $\Gamma=\langle S\mid R\rangle$ be a group with its presentation.
Then $\Gamma\isom \FF_S/R_N$, where $\FF_S$ is the free group with $S$ as a free generating set and $R_N$ is the normal subgroup generated by~$R$.
There is a canonical bijection between $R_N$ and the fundamental group $\pi_1(G)$ of the Cayley graph of~$\Gamma$ with respect to~$S$.
Via this bijection, every generating set for~$\pi_1(G)$ leads to a generating set for~$R_N$.
In particular, we obtain as a corollary of Theorem~\ref{thm_PiFinGen} Droms's theorem on the finite presentability of planar groups.

\begin{thm}\cite{D-InfiniteEndedGroups}
Every finitely generated planar group is finitely presented.\qed
\end{thm}

\section{Homology group of planar graphs}\label{sec_hom}

Instead of looking at the fundamental group, we consider in this section the first simplicial homology group $\HF_1(G)$ of graphs~$G$ as a module over~$\ganz$.
In particular, compared to the first section, the \emph{sum} of two cycles or closed walks is no longer dependent on the question \emph{where} we insert the first in the second one but just depends on the edge sets and the direction in which we pass the edges.
E.g., adding a spike does not change an element of the module and taking the inverse of a closed walk is just the same as taking the negative of the corresponding element of~$\HF_1(G)$.

Let $\VF$ be a finite set of closed walks and
\[
\VF'=\{E(W)\mid W\in\VF\},
\]
where $E(W)$ is the the of edges of the closed walk~$W$.
If $W$ is generated by~$\VF$ but by no proper subset of~$\VF$, then $E(W)$ is the sum of $\VF'$ with coefficients either $1$ or~$-1$.

So we can directly translate our results from the previous sections.
Another prossibility is to go through the proofs once more and see that it stays true with the new summation.

\begin{thm}
Let $G$ be a planar $3$-connected finitely separable graph.
Then $\HF_1(G)$ has a canonical nested generating set.\qed
\end{thm}

We call $\HF_1(G)$ a \emph{finitely generated $\Aut(G)$-module} if it has a generating set consisting of finitely many $\Aut(G)$-orbits.

\begin{thm}\label{thm_PlanarFiniteModule2}
Let $G$ be a locally finite planar graph.
Then $\HF_1(G)$ is a finitely generated $\Aut(G)$-module.\qed
\end{thm}

\section{Accessibility}\label{sec_Access}

A \emph{ray} is a one-way infinite path and two rays are \emph{equivalent} if they lie in the same component whenever we remove a finite vertex set.
This is an equivalence relation whose classes are the \emph{ends} of the graph.
We call a quasi-transitive graph \emph{accessible} if there is some $n\in\nat$ such that any two ends can be separated by removing at most~$n$ vertices.

The \emph{cycle space} of a graph~$G$ is the same as the first simplicial homology group except that we sum over $\mathbb{F}_2$ instead of~$\ganz$.
In~\cite{H-Accessibility} the author proved the following accessibility result for quasi-transitive graphs.

\begin{thm}\cite[Theorem~3.2]{H-Accessibility}\label{thm_AccessCited}
Every quasi-transitive graph $G$ whose cycle space is a finitely generated $\Aut(G)$-module is accessible.\qed
\end{thm}

As a corollary of Theorem~\ref{thm_PlanarFiniteModule2} together with Theorem~\ref{thm_AccessCited}, we obtain Dunwoody's theorem of the accessibility of locally finite quasi-transitive planar graphs, a strengthened version of Theorem~\ref{thm_accessMain}.
(Note that any generating set of the first homology group of a graph is also a generating set of its cycle space.)

\begin{thm}\label{thm_PlanarAccess}\label{thm_access}\cite{D-PlanarGraphsAndCovers}
Every locally finite quasi-transitive planar graph is accessible.\qed
\end{thm}

Note that, in order to prove Theorem~\ref{thm_PlanarAccess}, we do not need the full strength of a nested canonical generating set for the first homology group.
Indeed, instead of applying Theorem~\ref{thm_canGen}, we could just do the same arguments as in Section~\ref{sec_planar} using a nested canonical generating set for the cycle space obtained from~\cite[Theorem~1]{H-GeneratingCycleSpace} to obtain a finite set of cycles generating the cycle space as module.

\section*{Acknowledgement}

I thank J.~Carmesin for showing me a key idea of the proof of Theorem~\ref{thm_PlanarFiniteModule} and P.~Gollin for reading an earlier version of this paper.
I thank M.J.~Dunwoody for pointing out an error in an earlier version of this paper.

\providecommand{\bysame}{\leavevmode\hbox to3em{\hrulefill}\thinspace}
\providecommand{\MR}{\relax\ifhmode\unskip\space\fi MR }
\providecommand{\MRhref}[2]{%
  \href{http://www.ams.org/mathscinet-getitem?mr=#1}{#2}
}
\providecommand{\href}[2]{#2}

\end{document}